\documentclass[review]{elsarticle}

\usepackage{amssymb,amsmath,mathbbol}

\newtheorem{theorem}{Theorem}[section]

\newtheorem{proposition}[theorem]{Proposition}

\newdefinition{definition}[theorem]{Definition}
\newdefinition{example}[theorem]{Example}
\newproof{proof}{Proof}
\makeatletter
\def\ps@pprintTitle{%
 \let\@oddhead\@empty
 \let\@evenhead\@empty
 \def\@oddfoot{}%
 \let\@evenfoot\@oddfoot}
\makeatother

\begin{document}

\begin{frontmatter}

\title{A new characterization of the discrete Sugeno integral\footnote{Preprint of an article published by Elsevier in the Information Fusion 29 (2016), 84-86. It is available online at: \newline www.sciencedirect.com/science/article/pii/S1566253515000810}}

\author[up]{Radom\'ir Hala\v{s}}
\ead{radomir.halas@upol.cz}

\author[stu,ost]{Radko Mesiar}
\ead{radko.mesiar@stuba.sk}

\author[up,sav]{Jozef P\'ocs}
\ead{pocs@saske.sk}

\address[up]{Department of Algebra and Geometry, Faculty of Science, Palack\'y University Olomouc, 17. listopadu 12, 771 46 Olomouc, Czech Republic}
\address[stu]{Department of Mathematics and Descriptive Geometry, Faculty of Civil Engineering, Slovak University of Technology in Bratislava, Radlinsk\'eho 11, 810 05 Bratislava 1, Slovakia}
\address[ost]{University of Ostrava, Institute for Research and Applications of Fuzzy Modeling, NSC Centre of Excellence IT4Innovations, 30. dubna 22, 701 03 Ostrava, Czech Republic}
\address[sav]{Mathematical Institute, Slovak Academy of Sciences, Gre\v s\'akova 6, 040 01 Ko\v sice, Slovakia}

\begin{abstract}
We introduce a new property of the discrete Sugeno integrals which can be seen as their characterization, too. This property, compatibility with respect to congruences on $[0,1]$, stresses the importance of the Sugeno integrals in multicriteria decision support as well.
\end{abstract}


\end{frontmatter}

\section{Introduction and preliminaries}\label{sec:1}

Various integral-based operators have become a powerful tool in several applications, mainly in decision theory cf. \cite{Fus1,D3,D1,D2,Fus2,Fus3,Fus4}.
One of the important operators is represented by Sugeno integral, introduced in 1974 by Sugeno \cite{Sug74}, which has become a powerful tool in several applications rather soon \cite{D3,D1,D2}.

Among several properties of the Sugeno integrals, the crucial role is played by the fact that they are aggregation functions. Recall that a mapping $A\colon [0,1]^n\to [0,1]$ is called an aggregation function whenever it is nondecreasing in each coordinate, and satisfies two boundary conditions $A(\mathbf{0})=A(0,\dots,0)=0$ and $A(\mathbf{1})=A(1,\dots,1)=1$. For more details concerning aggregation functions we recommend the monographs \cite{BPC} and \cite{Grabisch et al 2009}. 

For $n\in\mathbb{N}$, denote $N=\{1,\dots,n\}$. The set function $m\colon 2^N\to [0,1]$ satisfying $m(\emptyset)=0$, $m(N)=1$ and being monotone, i.e., $m(E_1)\leq m(E_2)$ whenever $E_1\subseteq E_2\subseteq N$, is called a capacity (in some sources, e.g. in \cite{Sug74}, $m$ is called a fuzzy measure). The discrete Sugeno integral $\mathsf{Su}_m\colon [0,1]^n\to [0,1]$ with respect to a capacity $m$ is given by 

\begin{equation}\label{eq1}
\mathsf{Su}_m(\mathbf{x})=\bigvee_{t\in[0,1]}\big( t\wedge m(\{i\in N\mid x_i\geq t\})\big). 
\end{equation}

Two equivalent expressions defining the Sugeno integral $\mathsf{Su}_m$ are

\begin{equation}\label{eq2}
\mathsf{Su}_m(\mathbf{x})=\bigvee_{\emptyset\neq I\subseteq N}\big( m(I)\wedge (\bigwedge_{i\in I}x_i)\big)
\end{equation}
and
\begin{equation}\label{eq3}
\mathsf{Su}_m(\mathbf{x})=\bigvee_{i=1}^{n}\big( x_{(i)}\wedge m(\{(i),\dots,(n)\})\big), 
\end{equation}
where $(\cdot)\colon N\to N$ is a permutation such that $x_{(1)}\leq\ldots\leq x_{(n)}$. Recall that, for any capacity $m$, the Sugeno integral is

\begin{itemize}
\item[--] idempotent, $\mathsf{Su}_m(c,\dots,c)=c$ for any $c\in [0,1]$;
\item[--] comonotone maxitive, $\mathsf{Su}_m(\mathbf{x}\vee \mathbf{y})=\mathsf{Su}_m(\mathbf{x})\vee \mathsf{Su}_m(\mathbf{y})$ whenever $\mathbf{x}$ and $\mathbf{y}$ are comonotone, i.e., $(x_i-x_j)(y_i-y_j)\geq 0$ for any $i,j\in N$;
\item[--] min-homogeneous, $\mathsf{Su}_m(\mathbf{x}\wedge (c,\dots,c))=\mathsf{Su}_m(\mathbf{x})\wedge c$ for any $c\in [0,1]$. 
\end{itemize}

Several other properties of $\mathsf{Su}_m$ can be found in \cite{Wang} or in \cite{Couceiro}. Moreover, we have several kinds of axiomatization of the Sugeno integrals \cite{Benvenuti,CM2009LINZ,Couceiro,Marichal2000}. For example, the next claims are equivalent for an aggregation function $A\colon [0,1]^n\to [0,1]$.
\begin{itemize}
\item[i)] $A$ is comonotone maxitive and min-homogeneous;
\item[ii)] $A$ is a Sugeno integral, i.e., there is a capacity $m$ so that $A=\mathsf{Su}_m$ (observe that $m$ is given by $m(E)=A(1_E)$, where $1_E$ is the characteristic function of the set $E$, $1_E(i)=1$ if $i\in E$ and $1_E(i)=0$ otherwise);
\item[iii)] $A$ is median decomposable, i.e., for any $k\in N$, $\mathbf{x}\in [0,1]^n$ it holds $A(\mathbf{x})=\mathrm{med}\big(A(\mathbf{x}_k),x_k,A(\mathbf{x}^k)\big)$, where $(\mathbf{x}_k)_i=(\mathbf{x}^k)_i=x_i$ for $i\neq k$ and $(\mathbf{x}_k)_k=0$, $(\mathbf{x}^k)_k=1$.
\end{itemize}

Due to \cite{Couceiro} it is known that Sugeno integrals are just weighted lattice polynomials idempotent in $0$ and $1$. This fact, considering the real unit interval as a lattice $([0,1],\wedge,\vee)$, has inspired us to look for some other lattice techniques bringing new views on the Sugeno integrals. Namely, we have looked on lattice congruences and compatible aggregation functions, i.e., aggregation functions preserving each lattice congruence. Surprisingly, a new characterization of the Sugeno integrals was obtained. Moreover, the applicability of the Sugeno integral in multicriteria decision support was better clarified. 

This short note is organized as follows. In the next section, compatible $n$-ary aggregation functions are studied and their link to Sugeno integrals is shown. In Section \ref{sec:3}, the impact of our results to multicriteria decision support is discussed. Finally, some concluding remarks are added.

\section{Congruences on $[0,1]$ and compatible aggregation functions}\label{sec:2}

Recall that roughly, a congruence $C$ on an algebraic structure is an equivalence relation preserving the basic operations in the considered structure. The exact definition in case of lattices is as follows:

\begin{definition}
Let $(L,\wedge,\vee)$ be a lattice. A subset $C\subseteq L^2$ is called a congruence on $L$ whenever it is an equivalence relation preserving the join $\vee$ and the meet $\wedge$, i.e., 
\begin{enumerate}[i)]
\item $(a,a)\in C$ for each $a\in L$;
\item if $(a,b)\in C$ then also $(b,a)\in C$;
\item if $(a,b),(b,c)\in C$ then also $(a,c)\in C$;
\item if $(a,b),(c,d)\in C$ then also $(a\vee c,b\vee d)\in C$;
\item if $(a,b),(c,d)\in C$ then also $(a\wedge c,b\wedge d)\in C$.
\end{enumerate}
\end{definition}

For more details concerning lattice congruences and the next result we recommend \cite{G1}. For the sake of selfcontainedness, we prove the following proposition.

\begin{proposition}\label{prop1}
The following are equivalent:
\begin{enumerate}[i)]
\item $C$ is a congruence on $([0,1],\wedge,\vee)$;
\item there is a partition $\{J_k\mid k\in\mathcal{K}\}$ of $[0,1]$ such that $J_k$ is an interval for each $k\in\mathcal{K}$, and $C=\bigcup\limits_{k\in\mathcal{K}}J_k^2$.
\end{enumerate}
\end{proposition}

Let us note that each singleton is also considered as an interval of $[0,1]$.

\begin{proof}

$\mathrm{i)} \Rightarrow \mathrm{ii)}$ 
For any $x\in[0,1]$, denote $J_x=\{y\in[0,1]\mid (x,y)\in C\}$, i.e., $J_x$ is the $C$-equivalence class containing the element $x$. For any $u,v\in J_x$, $u<v$, and any $z\in ]u,v[$, it holds $(u,v),(z,z)\in C$ and thus $(u\vee z, v\vee z)=(z,v)\in C$, and $(u\wedge z,v\wedge z)=(u,z)\in C$. Thus $z\in J_v=J_u=J_x$, i.e., $J_x$ is an interval. Now, it is enough to define $k_x$ as a mid-point of $J_x$, $\mathcal{K}=\{k_x\mid x\in [0,1]\}$ and obviously $C=\bigcup_{k\in\mathcal{K}}J_k^2$, where $\{J_k\mid k\in\mathcal{K}\}$ is an interval-partition of $[0,1]$.

$\mathrm{ii)} \Rightarrow \mathrm{i)}$ is a matter of direct checking only. 
\qed


\end{proof}

Our main interest is to characterize aggregation functions preserving the congruences on $[0,1]$. 

\begin{definition}
An aggregation function $A\colon [0,1]^n\to [0,1]$ is called compatible if it preserves any congruence $C$ of $([0,1],\wedge,\vee)$, i.e., if for any $\mathbf{x},\mathbf{y}\in [0,1]^n$ such that $(x_1,y_1),\dots,(x_n,y_n)\in C$ also $\big(A(\mathbf{x}),A(\mathbf{y})\big)\in C$.
\end{definition}

\begin{theorem}\label{thm1}
Let $A\colon [0,1]^n\to [0,1]$ be an aggregation function. Then $A$ is compatible if and only if it is a Sugeno integral, i.e., if $A=\mathsf{Su}_m$ for some capacity $m$.
\end{theorem}

\begin{proof}
Let $A$ be a compatible aggregation function. For $\mathbf{x}\in [0,1]^n$ and $k\in N$, denote $\delta=A(\mathbf{x})$, $\alpha=A(\mathbf{x}_k)$ and $\beta=A(\mathbf{x}^k)$. Due to the monotonicity of $A$, obviously $\alpha\leq\delta\leq\beta$. Suppose that $x_k<\delta$ and consider a partition $\big\{[0,x_k],]x_k,1]\big\}$ of $[0,1]$. Clearly, we have $(\mathbf{x})_i=(\mathbf{x}_k)_i=x_i$ for $i\neq k$, and $(\mathbf{x})_k=x_k\in [0,x_k]$, $(\mathbf{x}_k)_k=0\in [0,x_k]$.
Due to the compatibility of $A$ we obtain $A(\mathbf{x}), A(\mathbf{x}_k)\in ]x_k,1]$ i.e., $\alpha>x_k$. Further, consider the partition $\big\{ [0,x_k],\{x\}\mid x\in ]x_k,1]\big\}$. Again, due to the compatibility of $A$, it holds $A(\mathbf{x}),A(\mathbf{x}_k)\in \{x\}$ for some $x\in ]x_k,1]$, i.e., $\delta=\alpha$. Hence, $x_k<\alpha\leq \beta$, and $$A(\mathbf{x})=\delta=\mathrm{med}(x_k,\alpha,\beta)=\mathrm{med}\big(x_k,A(\mathbf{x}_k),A(\mathbf{x}^k)\big).$$

Using a similar reasoning, we obtain in the case $x_k >\delta$ the equality $\delta=\beta$, inequalities $\alpha\leq\beta < x_k$ and thus $A(\mathbf{x})=\mathrm{med}(x_k,\alpha,\beta)=\mathrm{med}\big(x_k,A(\mathbf{x}_k),A(\mathbf{x}^k)\big)$.

Finally, if $x_k=\delta$, due to the above mentioned fact $\alpha\leq\delta\leq\beta$ we obtain $A(\mathbf{x})=\mathrm{med}\big(x_k,A(\mathbf{x}_k),A(\mathbf{x}^k)\big)$ as well. Summarizing all three discussed cases, we have shown the median decomposability of $A$. As $A(\mathbf{0})=0$ and $A(\mathbf{1})=1$, due to \cite{CM2009LINZ} it follows that $A$ is a Sugeno integral, $A=\mathsf{Su}_m$ for some capacity $m$.
\medskip

Conversely, suppose $A=\mathsf{Su}_m$ for some capacity $m\colon 2^N\to [0,1]$. Let $C$ be an arbitrary congruence on $([0,1],\wedge,\vee)$. Consider $\mathbf{x},\mathbf{y}\in [0,1]^n$ such that $(x_i,y_i)\in C$ for $i=1,\dots,n$. The compatibility of the meet $\wedge$ (for any arity) ensures that $\big(m(I)\wedge(\bigwedge_{i\in I}x_i),m(I)\wedge(\bigwedge_{i\in I}y_i)\big)\in C$ for any subset $\emptyset\neq I\subseteq N$ (recall that trivially $\big(m(I),m(I)\big)\in C$ holds).

Similarly, the join $\vee$ is compatible for any arity, and thus 
$$ \big( \bigvee_{\emptyset\neq I\subseteq N}\big( m(I)\wedge (\bigwedge_{i\in I}x_i)\big), \bigvee_{\emptyset\neq I\subseteq N}\big( m(I)\wedge (\bigwedge_{i\in I}y_i)\big)\big)\in C,$$
i.e., based on formula \eqref{eq2}, $\big( \mathsf{Su}_m(\mathbf{x}),\mathsf{Su}_m(\mathbf{y})\big)\in C$, proving the compatibility of $\mathsf{Su}_m$ with respect to $C$. Due to the arbitrariness of the choice of the congruence $C$, the compatibility of $\mathsf{Su}_m$ follows.
\qed
\end{proof}

Observe that Theorem \ref{thm1} can be straightforwardly formulated and validated for any bounded chain $(L,\wedge,\vee)$, i.e., an aggregation function $A\colon L^n\to L$ is compatible if and only if $A=\mathsf{Su}_m$ for some $L$-valued capacity $m\colon 2^N\to L$.

\section{Scale invariant aggregation functions}\label{sec:3}

Consider a bounded chain $L$ as a scale for the score in multicriteria decision linked to $n$ criteria. Then the normed utility function $A\colon L^n\to L$ can be seen as an aggregation function on $L$. For a bounded chain $L_1$, the epimorphism $\varphi\colon L\to L_1$ is a surjective homomorphism, i.e., $L_1=\{\varphi(x)\mid x\in L\}$, and for any $x,y\in L$, $x\leq_{L} y$, it holds $\varphi(x)\leq_{L_1}\varphi(y)$. 
Moreover, $\{\varphi^{-1}(\{a\})\mid a\in L_1\}$ is an interval partition of $L$. Obviously, if $\varphi\colon L\to L_1$ is an epimorphism then $\mathrm{card}(L)\geq\mathrm{card}(L_1)$.
Let us note, that it can be easily verified that in this case $\varphi$ is a lattice homomorphism, i.e., it preserves the lattice operations $\wedge$ and $\vee$.

\begin{example}\label{ex1}
As a particular example consider the real unit interval $L=[0,1]$.
\begin{itemize}
\item[(i)] The decimal half up rounding is related to the scale $L_1=\{0,0.1,\dots,0.9,1\}$, the corresponding epimorphism $\varphi\colon L\to L_1$ is given by the interval partition $\mathcal{P}=\big\{ \varphi^{-1}(0),\varphi^{-1}(0.1),\dots,\varphi^{-1}(1)\big\}$\\ $=\big\{[0,0.05[,\; [0.05,0.15[,\; \dots,\; [0.85,0.95[,\; [0.95,1] \big\}$.

\item[(ii)] Similarly, the centesimal half up rounding is related to \\ $L_1=\{0,0.01,\dots,0.99,1\}$ and \\ $\mathcal{P}=\big\{ [0,0.005[,\; [0.005,0.015[,\;\dots,\; [0.985,0.995[,\; [0.995,1]\big\}$.

\item[(iii)] For the linguistic scale $L_1=\{\mbox{bad},\mbox{ medium},\mbox{ good},\mbox{ excellent}\}$ one can consider, e.g., $\mathcal{P}=\big\{ [0,0.3],\;]0.3,0.7[,\;[0.7,0.9[,\;[0.9,1]\big\}$.
\end{itemize}
Having an aggregation function $A$ on $L$, its desirable property is the compatibility with the above mentioned epimorphisms.
\end{example}

\begin{definition}\label{def3}
Let $(L,\leq_L)$ be a bounded chain, and let $A\colon L^n\to L$ be an aggregation function on $L$. Then $A$ is called scale invariant whenever for any bounded chain $(L_1,\leq_{L_1})$ and epimorphism $\varphi\colon L\to L_1$ there is an aggregation function $B\colon {L_1}^n\to L_1$ such that for each $\mathbf{x}\in L^n$ it holds
\begin{equation}\label{eq4}
\varphi\big(A(\mathbf{x})\big)=B\big(\varphi(\mathbf{x}),\dots,\varphi(\mathbf{x})\big).
\end{equation}
\end{definition}

The next characterization of scale invariant aggregation functions on $L$ follows from Theorem \ref{thm1} (adapted for bounded chains).

\begin{theorem}\label{thm2}
Let $(L,\leq_L)$ be a bounded chain, and let $A\colon L^n\to L$ be an aggregation function on $L$. Then $A$ is scale invariant if and only if $A=\mathsf{Su}_m$, where the $L$-valued capacity $m\colon 2^N\to L$ is given by $m(I)=A(1^L_I)$, where $1^L_I(i)=1_L$ if $i\in I$ and $1^L_I(i)=0_L$ otherwise.
\end{theorem}

\begin{proof}
To see the necessity, it is enough for any epimorphism $\varphi\colon L\to L_1$, to put $B\colon {L_1}^n\to L_1$, $B=\mathsf{Su}_{\varphi(m)}$, where the $L_1$-valued capacity $\varphi(m)\colon 2^N\to L_1$ is given by $\varphi\big(m\big)(I)=\varphi\big(m(I)\big)$. Since $\varphi$ preserves the lattice operations, for each $\mathbf{x}=(x_1,\dots,x_n)\in L^n$ we obtain
$$ \varphi\big(\bigvee_{\emptyset\neq I\subseteq N}\big( m(I)\wedge (\bigwedge_{i\in I}x_i)\big)\big)=\bigvee_{\emptyset\neq I\subseteq N}\big(\varphi\big(m(I)\big)\wedge \big(\bigwedge_{i\in I}\varphi(x_i)\big)\big),$$
which yields $\varphi\big(\mathsf{Su}_m(\mathbf{x})\big)=\mathsf{Su}_{\varphi(m)}\big(\varphi(x_1),\dots,\varphi(x_n)\big)$.
\medskip

Suppose now that $A$ is scale invariant. For any congruence $C$ linked to an interval partition $\mathcal{P}$ of $L$, put $L_1=\mathcal{P}$, and $\leq_{L_1}$ given by $P_1\leq_{L_1} P_2$, $P_1,P_2\in \mathcal{P}$, whenever there are $x\in P_1$ and $y\in P_2$ such that $x\leq_L y$. Obviously, $\varphi\colon L\to L_1$ given by $\varphi(x)=P\in\mathcal{P}$, $x\in P$, is an epimorphism linked to the partition $\mathcal{P}$. Due to \eqref{eq4}, there is an aggregation function $B\colon {L_1}^n\to L_1$ such that, for each $\mathbf{x}\in L^n$, $\varphi\big(A(\mathbf{x})\big)=B\big(\varphi(\mathbf{x}),\dots,\varphi(\mathbf{x})\big)$.

However, then for any $\mathbf{x},\mathbf{y}\in L^n$ such that $(x_1,y_1),\dots,(x_n,y_n)\in C$ it holds $\varphi(x_1)=\varphi(y_1),\dots,\varphi(x_n)=\varphi(y_n) $ and thus 
$$ \varphi\big(A(\mathbf{x})\big)=B\big(\varphi(x_1),\dots,\varphi(x_n)\big)=B\big(\varphi(y_1),\dots,\varphi(y_n)\big)=\varphi\big(A(\mathbf{y})\big),$$
yielding $\big(A(\mathbf{x}),A(\mathbf{y})\big)\in C$. Hence, $A$ is a compatible aggregation function on $L$, i.e., $A=\mathsf{Su}_m$ for an $L$-valued capacity $m$.
\qed
\end{proof}

Observe that for scale invariant aggregation function $A=\mathsf{Su}_m\colon L^n\to L$, and any epimorphism $\varphi\colon L\to L_1$, the corresponding aggregation function $B\colon {L_1}^n\to L_1$ linked to $A$ by formula \eqref{eq4} is also a Sugeno integral, $B=\mathsf{Su}_{\varphi(m)}$.

\begin{example}\label{ex2}
Continuing in Example \ref{ex1}, consider $n=3$ and a capacity $m$ given by $m(I)=\frac{\mathrm{card}(I)}{3}$. For $\mathbf{x}=(0.54,\frac{1}{\sqrt{2}},\frac{3}{7})=(x_1,x_2,x_3)$ it holds $x_3<x_1<x_2$ and $\mathsf{Su}_m(\mathbf{x})=(x_3\wedge 1)\vee (x_1\wedge \frac{2}{3})\vee (x_2\wedge \frac{1}{3})=0.54$ by formula \eqref{eq3}.
\begin{itemize}
\item[(i)] In the case of decimal half up rounding, it holds 
$$\varphi\big(m(I)\big)=\begin{cases}0 &\mbox{ if }I=\emptyset, \\
                                     0.3 &\mbox{ if }\mathrm{card}(I)=1,\\  
                                     0.7 &\mbox{ if }\mathrm{card}(I)=2, \\
                                     1 &\mbox{ if }\mathrm{card}(I)=3,\end{cases}$$
and $\varphi(\mathbf{x})=(0.5,0.7,0.4)$. Then 
$$ \mathsf{Su}_{\varphi(m)}\big(\varphi(\mathbf{x})\big)=(0.4\wedge 1)\vee(0.5\wedge 0.7)\vee(0.7\wedge 0.3)=0.5=\varphi(0.54);$$                                                              

\item[(ii)] In the case of centesimal half up rounding,   
$$\varphi\big(m(I)\big)=\begin{cases}0 &\mbox{ if }I=\emptyset, \\
                                     0.33 &\mbox{ if }\mathrm{card}(I)=1,\\  
                                     0.67 &\mbox{ if }\mathrm{card}(I)=2, \\
                                     1 &\mbox{ if }\mathrm{card}(I)=3,\end{cases}$$  
and $\varphi(\mathbf{x})=(0.54,0.71,0.43)$. Then
$$ \mathsf{Su}_{\varphi(m)}\big(\varphi(\mathbf{x})\big)=(0.43\wedge 1)\vee(0.54\wedge 0.67)\vee(0.71\wedge 0.33)=0.54=\varphi(0.54);$$                                                                
\item[(iii)] In the case of linguistic scale $L_1$ introduced in Example \ref{ex1} (iii),  
$$\varphi\big(m(I)\big)=\begin{cases}\mbox{bad} &\mbox{ if }I=\emptyset, \\
                                     \mbox{medium} &\mbox{ if }\mathrm{card}(I)\in\{1,2\},\\  
                                     \mbox{excellent} &\mbox{ if }\mathrm{card}(I)=3, \end{cases}$$
and $\varphi(\mathbf{x})=(\mbox{medium},\,\mbox{good},\, \mbox{medium})$. Then
$$ \mathsf{Su}_{\varphi(m)}\big(\varphi(\mathbf{x})\big)=(\mbox{medium} \wedge \mbox{excellent})\vee (\mbox{medium} \wedge \mbox{medium} )\vee $$
$$(\mbox{good} \wedge \mbox{medium})=\mbox{medium} =\varphi(0.54).$$
\end{itemize}

\end{example}
\section{Concluding remarks}

We have completely characterized compatible aggregation functions acting on an (arbitrary) bounded chain $L$, and in particular on $[0,1]$. These functions coincide with Sugeno integrals with respect to ($L$-valued) capacities, and thus the compatibility is a new characteristic property of Sugeno integrals. Moreover, we have shown that the scale invariant normed utility functions are just Sugeno integrals, and thus the importance of Sugeno integrals acting as utility functions on ordinal scales was stressed.

\section*{Acknowledgements}
The first author was supported by the international project Austrian Science Fund (FWF)-Grant Agency of the Czech Republic (GA\v{C}R) number I 1923-N25; the second author by the Slovak Research and Development Agency under contract APVV-0013-14 and by the European Regional Development Fund in the IT4Innovations Centre of Excellence project reg. no. CZ.1.05/1.1.00/02.0070; the third author by the ESF Fund CZ.1.07/2.3.00/30.0041 and by the Slovak VEGA Grant 2/0028/13.

\end{document}